\documentclass[11pt]{article}
\usepackage{graphicx}
\usepackage{amsthm}
\usepackage{amssymb}
\usepackage{amsmath}
\usepackage{float}
\usepackage{hyperref}
\usepackage{graphicx}
\usepackage{subcaption}
\usepackage{ mathrsfs }
\usepackage[font=small,labelfont=bf]{caption}
\usepackage{algorithm}
\usepackage{algpseudocode}
\graphicspath{ {./images/} }

\usepackage[left=30mm,right=30mm,top=20mm,bottom=20mm]{geometry}

\usepackage{tikz}
\tikzset{every picture/.style={line width=0.75pt}}

\usepackage{comment}

\theoremstyle{plain}
\newtheorem{theorem}{Theorem}[section]
\newtheorem{corollary}[theorem]{Corollary}
\newtheorem{lemma}[theorem]{Lemma}

\newtheorem{proposition}[theorem]{Proposition}

\theoremstyle{definition}

\theoremstyle{remark}
\newtheorem*{remark}{Remark}

\title{On the $Z_q$-forcing number:\\ computational approach and exact values}

\author{Aida Abiad\thanks{\texttt{a.abiad.monge@tue.nl},  Department of Mathematics and Computer Science, Eindhoven University of Technology, The Netherlands}\thanks{Department of Mathematics and Data Science of Vrije Universiteit Brussel, Belgium}
\qquad Maryam Moghaddas \thanks{\texttt{sanazmoghaddas1380@gmail.com}, 
Department of Mathematical Sciences, Sharif University of Technology, Iran}
}

\date{}

\begin{document}

\maketitle

\begin{abstract}
Zero forcing is a graph coloring process that is used to model spreading phenomena in real-world scenarios. It can also be viewed as a single-player combinatorial game on a graph, where the player's goal is to select a subset of vertices of minimum cardinality that eventually leads to all vertices of the graph being colored. A variant of this game, called the $q$-analogue of zero forcing, was later introduced. In this version, the player again seeks to choose the smallest number of vertices that will eventually color the entire graph, while an oracle attempts to force the player to select a larger subset.

In this paper, we exploit the structural properties of several graph classes in order to both derive algorithms to compute the exact value of $Z_q$, and to establish bounds and exact values of the parameter for these graph classes. In particular, we present a SAT-based algorithm to compute the $q$-analogue zero forcing number, offering optimal strategies for both the player and the oracle. Additionally, we propose a polynomial-time algorithm for computing the $q$-analogue zero forcing number for $q=1$ of cactus graphs. Lastly, we prove the exact value of this parameter for several graph classes such as block graphs. Our work extends previous results about trees by Butler et al. (2020) and Blanco et al. (2024).\\

\end{abstract}

\section{Introduction}

In this paper, we focus on zero forcing, a combinatorial single-player game played on a graph, whose objective is to color all vertices of the graph starting from a subset of minimum size. Equivalently, one may view the process as the player placing tokens on a subset of vertices to color them, and then propagating the coloring to the remaining vertices according to the zero forcing rules. The player can employ two rules: spending a token to fill a vertex, and applying the forcing rule. More formally, the rules are as follows:  
\begin{description}
    \item[Rule 1:] Filling any vertex at the cost of one token.
    \item[Rule 2:] Applying the \emph{forcing rule} at no cost (a filled vertex with a unique unfilled neighbor forces that neighbor to be filled).
\end{description}  
The minimum number of tokens required to fill all vertices in a graph $G$ is denoted by $Z(G)$ and is referred to as the \emph{zero forcing number} of $G$.

The forcing process can be regarded as a specific type of spreading mechanism on graphs, extensively studied in mathematics and computer science. For instance, it was proposed in \cite{PhysRevLett.99.100501} 
and was later applied in \cite{PhysRevA.79.060305} as a tool for establishing quantum controllability, and it is used as a bound for the minimum rank (or equivalently, the maximum nullity) of a graph \cite{AIMMINIMUMRANKSPECIALGRAPHSWORKGROUP20081628}. Zero forcing also finds applications in areas such as physics, statistical mechanics, and social network analysis, where graph-based processes are used to model technical phenomena or social behavior. See \cite{dynamicalsystemsZ} for an overview of these models and applications. Robustness and resilience of complex networks, which has been investigated in the literature by several authors (see e.g. \cite{AGDetal2024, JI20231}), has also been approached using a certain forcing process in graphs \cite{A2023}.

In the present paper we will be mainly concerned with a variant of this game known as the $q$-analogue of the zero forcing, which was introduced by Butler at al. in \cite{butler2015using} and studied in \cite{butler2020properties, fallat2024q,blanco2023z_q}. The $q$-analogue of zero forcing (or $Z_q$ game) introduces a third rule and transforms the game into a two-player setting between a player and an oracle. This additional rule allows the player to apply the forcing rule within an induced subgraph of the original graph, potentially reducing the amount of required tokens. Let $F$ denote the set of filled vertices, and let $U_1, U_2, \dots, U_k$ be the connected components of $G[V \setminus F]$ (the subgraph induced by the unfilled vertices). The parameter $q$ is a fixed integer constant. The new rule is as follows:

\begin{description}
    \item[Rule 3:] If $k > q$, the player selects at least $q+1$ of the components $U_i$. The oracle then chooses a nonempty subset of these components, denoted $\{U_{i_1}, U_{i_2}, \dots, U_{i_l}\}$. The player is then allowed to apply the forcing rule on the subgraph induced by $F \cup U_{i_1} \cup U_{i_2} \cup \cdots \cup U_{i_l}$.
\end{description}  

The \emph{$q$-analogue zero forcing number}, denoted as $Z_q(G)$, is the minimum tokens needed to fill $G$, assuming an adversarial oracle that maximizes this requirement.  Thus, while zero forcing is a combinatorial game on graphs where a player aims to fill all vertices using the minimum number of tokens following specific rules, the $q$-analogue of zero forcing generalizes this game by introducing an adversarial oracle, enabling the player to operate on smaller induced subgraphs. Note that the parameter $Z_q$, when was introduced by Butler et al. \cite{butler2015using} was called  \textit{zero forcing with $q$ negative eigenvalues} and was proposed as a generalization of both zero forcing and semi-definite zero forcing by allowing the player to operate on a smaller graph with multiple components.

Although the zero forcing game and its $q$-analogue game share similarities, the presence of the oracle makes the latter more complex, and consequently multiple zero forcing techniques do not naturally extend. For instance, as mentioned in \cite{butler2015using}, unlike in zero forcing, Rule 1 moves cannot always precede Rule 2 moves, as this may advantage the oracle and increase the required tokens. In zero forcing, all tokens can be assumed to be placed at the start, with Rule 2 applied only after Rule 1, which has greatly shaped computational approaches. Consequently, while several computational techniques exist for zero forcing \cite{bbfh2019, bmh2021, aazami2008hardness, A2010}, less is known for the $q$-analog version of the zero forcing game.

In this paper, we exploit structural properties of several graph classes in order to both derive algorithms to compute the exact value of $Z_q$ and to establish bounds for certain graph classes. In particular, we present polynomial-time algorithms for $Z_0$ in cactus graphs and for $Z_q$ in block graphs, thus extending known results for trees by Butler et al. \cite[Theorem 6]{butler2020properties}, who proposed a polynomial-time algorithm for computing $Z_1$ in trees. Despite the fact that $Z_q$ is probably an NP-hard parameter (its complexity remains open, to the best of our knowledge), its exact value has been determined for certain graph classes. Instances of it are the work of Butler et al. \cite{butler2020properties} (exact values of $Z_0$ for trees and $Z_q$ for paths) and Fallat et al. \cite{fallat2024q} (exact values of $Z_q$ for complete bipartite graphs, strongly regular graphs, and some graphs in a certain association scheme). Blanco et al. \cite{blanco2023z_q} provided bounds for caterpillar cycles, an upper bound for $Z_q$ in trees, and exact values for forests of star graphs. We should note that Blanco et al. approached the study of $Z_q$ from an adversarial perspective, modeling the player and the oracle as opponents and developing optimal strategies for both. In this paper we use a similar approach to show the value of $Z_q$ for block graphs, generalized star graphs, and generalized windmill graphs, providing optimal strategies for both the player and the oracle. 

This paper is structured as follows. In Section \ref{sec:preliminaries} we set up some preliminaries on zero forcing. In Section \ref{sec:computational} we explore computational approaches to the $q$-analogue zero forcing, including a SAT-based algorithm for computing $Z_q$, polynomial-time algorithms for $Z_0$ in cactus graphs, and $Z_q$ in block graphs. Finally, in Section \ref{sec:exactvaluegraphclasses} we examine generalized star graphs, block graphs, and windmill graphs, providing exact values along with optimal strategies for both the player and the oracle.

\section{Preliminaries}\label{sec:preliminaries}

In this paper, we adopt standard graph-theoretic notation. For a graph $G$, let $V(G)$ denote its set of vertices and $E(G)$ denote its set of edges. The order of $G$ is $|V(G)|=n$ and its size is $|E(G)|=m$. For a vertex $v \in V(G)$, the open neighborhood is $N(v) = \{u \in V(G) : (u,v) \in E(G)\}$, and the closed neighborhood is $N[v] = N(v) \cup \{v\}$. A clique, or complete graph, denoted by $K_n$, is a graph in which every pair of vertices is adjacent. A complete bipartite graph is one whose vertex set can be partitioned into two disjoint sets $A$ and $B$ with $A \cup B = V(G)$ and $A \cap B = \emptyset$, such that every vertex in $A$ is adjacent to every vertex in $B$, and there are no edges within $A$ or within $B$. We denote it by $K_{a,b}$, where $a = |A|$ and $b = |B|$.

The \emph{zero forcing process} on a graph $G$, in its single-player game formulation, can be described as follows. Initially, the player selects a subset $S \subseteq V(G)$ and fills its vertices by spending one token on each (Rule 1). Subsequently, a vertex can be filled if it is the unique unfilled neighbor of some already filled vertex (Rule 2). The objective is to minimize the number of tokens required to eventually fill all vertices of the graph. A subset $S \subseteq V(G)$ is called a \emph{zero forcing set} if starting from $S$, all vertices of $G$ can be filled. The \emph{zero forcing number} of $G$, denoted by $Z(G)$, is the minimum cardinality of such a set.

The zero forcing number $Z(G)$ can be used to bound a graph-theoretic parameter called the \emph{maximum nullity}. For a graph $G$ with $n$ vertices, define $S(G)$ as the set of all $n \times n$ real symmetric matrices $B = [b_{ij}]$ such that $b_{ij} \neq 0 \Leftrightarrow \{i,j\}\in E(G)$. In other words, $G(B)$, the graph represented by $B$, is identical to $G$. Using this setup, one can define the \emph{maximum nullity} and the \emph{minimum rank} of $G$ as follows:

\vspace{-0.1cm}
\begin{align*}
   & M (G) = \max \{\text{null}(X):  X\in \mathcal{S}(G)\},\\
   & mr (G) = \min\{\text{rank}(X): X \in \mathcal{S}(G)\}.
\end{align*}
The bound provided by $Z(G)$ is that $M(G) \leq Z(G)$, or equivalently, $mr(G) \geq n-Z(G)$.

The \emph{positive semidefinite zero forcing number}, denoted by $Z_{+}(G)$, is a variant of the zero forcing number with a modified coloring rule: a vertex $v$ can be filled without spending a token if it is adjacent to a filled vertex $u$ and there is no path consisting entirely of unfilled vertices from any other unfilled neighbor of $u$ to $v$. The minimum size of a set $S \subseteq V(G)$ on which tokens must be spent to eventually fill all vertices without using additional tokens defines $Z_+(G)$. Further details, results, and related variants of zero forcing numbers and their associated parameters can be found in \cite[Chapters 2 \& 7--10]{hogben2022inverse}.

We note that the $q$-analog of zero forcing can be interpreted as a parameter that establishes a connection between the zero forcing number and the positive semidefinite forcing number. The following result shows this relationship.

\begin{proposition} \label{proposition:z+zqz}
\cite[Proposition 3.2]{butler2015using} For any graph $G$ with $n$ vertices,
$$
Z_+(G) = Z_0(G) \leq Z_1(G) \leq \dots \leq Z_n(G) = Z(G).
$$
\end{proposition}

\section{A computational approach to \texorpdfstring{\(Z_q\)}{Zq}}\label{sec:computational}

To the best of the authors' knowledge \cite{private}, determining the complexity of $Z_q$ is still an open problem. Nevertheless, several computational results have been obtained regarding the $q$-analogue variant of zero forcing. Butler et al. \cite{butler2015using} proposed an exponential-time algorithm for computing $Z_q$ for a given graph $G$ and parameter $q$, which exhaustively examines all subsets of the vertices in $G$. In this section we present a SAT-based algorithm for computing the $Z_q$-forcing number that also runs in an exponential-time. We should note that our approach differs significantly from the method proposed by Butler et al. \cite{butler2015using}, and it is inspired by the work of Brimkov et al. \cite{bmh2021}, who employed SAT modeling to enhance computational techniques for the classical zero forcing game. Thus, extending Brimkov et al. approach, we develop a SAT-based framework to compute the $Z_q$-forcing number. This method not only provides the value of the $Z_q$-forcing number but also determines optimal strategies for both the oracle and the player. Furthermore, Butler et al. \cite{butler2020properties} also introduced a polynomial-time algorithm for computing $Z_1$ for trees. Here, we extend their result by presenting a polynomial-time algorithm for cactus graphs.

\subsection{SAT model for computing \texorpdfstring{\(Z_q\)}{Zq}}\label{sec:SAT}

A strategy $ s $ for the oracle specifies, for a graph $ G $, the components chosen by the oracle when the vertices in set $ B $ are already colored, and the player’s selected components are $ C $. That is, for given sets $ B $ and $ C $, strategy $ s $ outputs a non-empty subset $ s_{B,C} \subseteq C $ (here, $C$ denotes the union of the vertices in the components, and $s_{B,C}$ denotes the union of the vertices in the chosen components). Note that if $C$ contains fewer than $q+1$ components, or if the vertices in $C$ are not the union of at least $q+1$ components, then we have $s_{B,C} = \emptyset$. We define $ S $ as the union of all possible strategies for the oracle. The number of such strategies is exponential, leading to the exponential time complexity of our algorithm. 

We define the Boolean variable $ \texttt{token}(v, i, s) $, which is true if the player spends a token on vertex $ v $, where $ v $ is the $ i $-th vertex to be colored and the oracle follows strategy $ s $. Similarly, the Boolean variable $ \texttt{forced}(v, i, s) $ is true if $ v $ is filled using Rule 2 or Rule 3, where $ v $ is the $ i $-th vertex filled, and the oracle follows strategy $ s $.

Moreover, the Boolean variable $ \texttt{feasible}(v, i, s) $ is true if, given the vertices that have already been colored up to step $ i $, there exists a subset of components such that, under oracle strategy $ s $, vertex $ v $ belongs to the oracle’s selected components and can be forced by a filled vertex; or, $v$ is the only unfilled neighbor of one of its neighbors in step $i$.

Our objective is to account for all possible strategies of the oracle and to determine a strategy for the player that minimizes the maximum number of tokens spent, regardless of the oracle's strategy. In this algorithm, every possible strategy for the oracle is considered. Assuming that the oracle follows a specific strategy $ s $, the player minimizes the tokens used. Consequently, the player determines a strategy to minimize the maximum tokens required across all oracle strategies.

Below is the SAT model for computing the $q$-analogue zero forcing number for a given number $q$ and a given graph $G$.
\begin{align}
    \texttt{minimize:} & \max_{s \in S} \sum_{v \in V}\sum_{i=0}^n \texttt{token}(v,i,s)\\
    \texttt{s.t.:} &\bigwedge_{v \in V} (\bigwedge_{s \in S}(((\bigvee_{i=0}^n \texttt{token}(v,i,s)) \wedge (\bigwedge_{i=0}^n \neg \texttt{forced}(v,i,s)) ) \vee \\
    &((\bigvee_{i=0}^n \texttt{forced}(v,i,s)) \wedge (\bigwedge_{i=0}^n \neg \texttt{token}(v,i,s)))))\\
    &\bigwedge_{v \in V}(\bigvee_{s \in S}(\bigvee_{i=0}^n((\texttt{token}(v,i,s) \vee \texttt{forced}(v,i,s))\\
    &\bigwedge_{j=0, j \neq i} ^ n (\neg \texttt{token}(v,j,s) \wedge \neg\texttt{forced}(v,j,s)))))\\
    &\bigwedge_{v \in V}(\bigwedge_{s \in S}(\bigwedge_{i=0}^n(\neg\texttt{forced}(v,i,s) \vee (\texttt{forced}(v,i,s) \wedge \texttt{feasible}(v,i,s)))))
\end{align}

The above SAT model minimizes the maximum number of tokens required across all possible strategies $s$, as specified in (1). Constraints (2) and (3) ensure that each vertex is either filled using a token or forced to be filled, but not both. Constraints (4) and (5) guarantee that every vertex is filled exactly once (in a unique step $i$). Finally, constraint (6) ensures that if a vertex is forced to be filled, it is indeed feasible to do so based on the strategy $s$ and the step at which the vertex is filled.

Next we present a SAT model for the variable $\texttt{feasible}(v,i,s)$, defined earlier.  
To improve clarity, we first introduce the auxiliary variable $\texttt{filled}(v,i,s)$, which indicates whether vertex $v$ is filled before the $i$th step under oracle strategy $s$ (either by spending a token or by being forced). Below is the SAT model for $\texttt{filled}(v, i, s)$:

\begin{align}
    \texttt{filled}(v,i,s) = \bigvee_{j=0}^{i-1} \big( \texttt{forced}(v,j,s) \vee \texttt{token}(v,j,s) \big).
\end{align}
Next, let $\beta = \{v \mid \texttt{filled}(v,i,s) = \text{true}\}$.  
We now provide the SAT model for the variable $\texttt{feasible}(v,i,s)$:

\begin{align}
    \texttt{feasible}(v,i,s) = &(\bigvee_{\substack{u \in N(v) \\ u \in \beta}} (\bigwedge_{\substack{w\in N(u) \\ u \neq v}}\texttt{filled}(w, i, s))) \vee \\
    &(\bigvee_{\substack{\mathcal{C} \in V\setminus\beta \\ v \in \mathcal{C} \\ s(\beta, \mathcal{C}) \neq \emptyset}}(\bigvee_{\substack{u\in N(v) \\ u \in \beta}} (\bigwedge_{\substack{w\in N(v) \\ w \in V \\ w \in s(\beta, \mathcal{C})}} \texttt{filled}(w,i,s) )))
\end{align}

Constraint (8) checks whether Rule~2 is applicable, i.e., whether there exists a neighbor of $v$ for which $v$ is its only unfilled neighbor.  
Constraint (9) considers all possible choices of unfilled components that contain $v$, and ensures that such a choice is feasible (by requiring $s(\beta, \mathcal{C}) \neq \emptyset$). It then verifies whether there exists a neighbor of $v$ whose only unfilled neighbor among the vertices selected by the oracle ($s(\beta, \mathcal{C})$) is precisely $v$.

\subsection{Polynomial algorithm for computing \texorpdfstring{\(Z_q\)}{Zq} for block graphs}\label{subsec:polyalgoblckgraphs}

A \textit{block graph}, is an undirected graph in which each maximal connected subgraph that cannot be disconnected by removing a single vertex forms a clique. We refer to these cliques as \textit{blocks}.

Later in Section \ref{subsec: blockgraphs} we will show that for block graphs whose block sizes are greater than two, $Z(G) = Z_q(G)$ for any value of $q$. Hence, in this section we can restrict ourselves to show an algorithm that computes $Z(G)$ for such block graphs (since the mentioned result will then provide an algorithm that computes their $Z_q(G)$ value).

In order to derive our polynomial algorithm to compute $Z$ in block graphs we need some preliminary lemmas.

Consider a block graph $G$ whose block sizes are at least three. For any block with size $\eta$ that shares only one of its vertices with other blocks, we must spend at least $\eta - 2$ tokens to completely fill this block. The next result makes this observation more precise.

\begin{lemma} \label{lemma:blockgrapheta}
Let $G$ be a block graph such that each block has at least three vertices in them. Consider a block that only share one of its vertices (such a block must exist). Let $\eta$ be the size of this block and $U$ be the set of its vertices. To be able to fill the whole graph in the zero forcing game, the player needs to spend at least $\eta - 2$ tokens on the vertices in this block.
\end{lemma}

\begin{proof}
Let $U$ be the set of vertices in this block, with $v \in U$ as the shared vertex and $u_1, u_2, \dots, u_{\eta -1}$ as the remaining vertices. Suppose the player spends $t < \eta - 2$ tokens on this block. Then, at least two vertices in $\{u_1, u_2, \dots, u_{\eta-1}\}$ must be filled using Rule 2. Assume $u_1$ and $u_2$ are the ones that must be filled this way. However, since $u_1$ and $u_2$ share the same neighborhood, any vertex $w$ connected to one of them must also be connected to the other. As both remain unfilled, the forcing rule cannot be applied to either. Thus, the player must spend at least $\eta - 2$ tokens on $U$.
\end{proof}.

We just have shown that the player must spend at least $\eta - 2$ tokens on a block of size $\eta$ with a single shared vertex. Supose that we have a block with only one shared vertex. The next result shows that either the player fills all but one of the vertices in that block using Rule 1, or no tokens are spent on the shared vertex.

\begin{lemma} \label{lemma:unfilledshared}
    Consider a block of size $\eta$ with a single shared vertex $v$. It is not possible to fill the entire graph if exactly $\eta - 2$ vertices of this block are filled using Rule 1, with $v$ being one of them. 
\end{lemma}

\begin{proof}
    Note that filling $\eta - 1$ vertices is sufficient to fill the entire block, and it does not matter which vertex remains unfilled using Rule 1. This is because there is always a non-shared vertex that is filled using a token, which can then force the remaining unfilled vertex to be filled. However, if exactly $\eta - 2$ vertices are filled using Rule 1 (recall that in \ref{lemma:blockgrapheta}, we showed it is not possible to fill fewer than $\eta - 2$ vertices using Rule 1 in such blocks), the shared vertex $v$ cannot be among them.  

    Assume, for contradiction, that in an optimal solution, $v$ is filled using Rule 1. Then, two other vertices in this block must be filled using Rule 2. Since $v$ is already filled, these two vertices must be non-shared vertices. Consequently, their neighborhood consists only of the vertices in the block. However, all filled vertices in this block are neighbors of both, meaning none of them can force either of these two vertices to be filled, leading to a contradiction.
\end{proof}

\begin{theorem} \label{Theorem: Z algorithm block graph}
Let $G$ be a block graph in which each block has size at least three. There exists a polynomial-time algorithm to compute $Z(G)$.
\end{theorem}

\begin{proof}
    
Our polynomial algorithm for block graphs is as follows. In each step, we consider a block graph whose blocks have size at least three, with some vertices already filled. We start with the original graph $G$ with no filled vertices. At each step, we identify a block that shares exactly one vertex, denoted by $v$, with other blocks. Suppose this block has $\eta$ vertices. There are two cases to consider:

\begin{itemize}
    \item If at least $\eta - 1$ of its vertices are already filled, we assume the entire block is filled. In this case, we remove this block (except for vertex $v$) and continue the process with the resulting smaller graph, treating $v$ as a filled vertex.  

    \item If fewer than $\eta - 1$ vertices are filled, we spend the necessary tokens to ensure that at most one vertex other than $v$ remains unfilled. If, after spending these tokens, only one vertex remains unfilled, we proceed as described above. Otherwise, by Lemma \ref{lemma:unfilledshared}, we know that $v$ and exactly one other vertex remain unfilled. Since filling either of them allows the other to be filled by the forcing rule, we remove the block (except for $v$), continue with the smaller graph, and treat $v$ as an unfilled vertex. This way, when $v$ is later filled in the reduced graph, we assume the remaining unfilled vertex in the removed block is filled via Rule 2.      
\end{itemize}

At each step, the graph is reduced by one block. We arbitrarily select a block that shares only one vertex with other blocks, spend tokens to reach the required state, and remove the block (except for the shared vertex). This process continues until the graph consists of a single vertex (which is filled by spending one token) or becomes empty.

Each round of this process requires at most $O(n^2)$ time, as we need to identify a block with a single shared vertex and check its vertices to determine which are filled and unfilled. Since at least two vertices are removed at the end of each round (as the removed block contains more than two vertices, excluding the shared vertex), the total number of rounds is $O(n)$. Thus, the entire process runs in $O(n^3)$ time, ensuring a polynomial time approach. We note that the complexity here is obtained using a naive analysis; in Appendix~\ref{apx: faster zq block graphs}, we have further optimized this complexity to $O(n)$ by efficiently selecting the block with only one shared vertex. However, our primary focus here was to demonstrate the existence of a polynomial time algorithm for computing $Z(G)$ for block graphs.
\end{proof}

\begin{proposition}
    The described algorithm uses an optimal number of tokens to fill the whole block graph.
\end{proposition}

\begin{proof}
    We prove this result by induction. According to our algorithm, to fill an empty graph, we spend no tokens; to fill a single vertex, we spend one token; and to fill a graph consisting of a single block with $n$ vertices, we spend $n - 1$ tokens, which is the optimal number of tokens required.

    Now, assume the result holds for block graphs with $b-1$ blocks. We aim to prove that it also holds for block graphs with $b$ blocks.

    Consider an optimal solution, denoted as $S$. By a solution, we mean a subset of vertices that must be filled using Rule 1 to eventually force all other vertices and fill the entire graph. Now, consider a block $B$ of size $\eta$ that has only one shared vertex, $v$. By Lemma \ref{lemma:blockgrapheta}, we know that in both the optimal solution and our algorithm, at least $\eta - 2$ vertices and at most $\eta - 1$ vertices in $B$ are filled using Rule 1. If in the optimal solution, all vertices in $B$ were filled by spending tokens, we reach a contradiction since, after $\eta - 1$ tokens are spent on this block, one arbitrary vertex could have been filled by the forcing rule, reducing the total number of spent tokens.

    Now, assume by contradiction that the optimal solution uses fewer tokens than our algorithm. If in the optimal solution $\eta - 1$ vertices are filled, we remove the block $B$ except for vertex $v$ and consider $v$ as a filled vertex in the new graph. Otherwise, we remove $B$ and treat $v$ as an unfilled vertex. In both cases, we have removed $\eta - 2$ filled vertices. The remaining graph, $G - B$, consists of $b-1$ blocks can be filled using $|S| - (\eta - 2)$ tokens by filling the vertices in $S \setminus (B \setminus \{v\}) $.

    In our algorithm, we spend $\eta - 2$ tokens on block $B$ and then solve the problem for the graph without $B$ (which now has $b-1$ blocks). If our algorithm spends $t$ tokens to fill the entire graph $G$, it follows that it spends $t - (\eta - 2)$ tokens to fill $G - B$.

    By the induction hypothesis, our solution for $b-1$ blocks is optimal. However, if our algorithm for the graph $G$ with $b$ blocks is not optimal, then $|S| < t$, which implies $|S| - (\eta - 2) < t - (\eta - 2)$. This contradicts the assumption that our algorithm provides an optimal solution for $b-1$ blocks.

Therefore, we have proved that our algorithm uses the optimal number of tokens.
\end{proof}

Appendix~\ref{apx: faster zq block graphs} contains the pseudocode for the above algorithm for block graphs, and the corresponding implementation in Python can be found in this \href{https://github.com/sanazmgh/Zq-Forcing-Number-Algorithms}{GitHub repository}.

\subsection{Polynomial algorithm for computing \texorpdfstring{\(Z_0\)}{Z0} for cactus graphs}\label{sec:polyalgocactusgraphs}

A \textit{cactus graph} is an undirected graph in which each edge belongs to at most one cycle. Equivalently, any two cycles in the graph share at most one vertex. By definition, a cactus graph can be viewed as a tree where each vertex represents either a cycle or a singular vertex not included in any cycle. By \textit{pendant cycle} we mean a cycle that shares exactly one vertex with another cycle or a vertex that is not part of any cycle and has exactly one neighbor. Notably, by the definition of a cactus graph, such a cycle or vertex always exists.

When $q=0$, the oracle has no strategic choices since it must always announce the only component it receives from the player. From Proposition \ref{proposition:z+zqz}, we know that $Z_0 = Z_+$. Unlike in the $q$-analogue of the zero forcing number (for $q > 0$), here we can first apply Rule 1 and then proceed with the other rules. Now, consider a cycle. All vertices in the cycle can be filled using Rule 3 as soon as two of its vertices are filled. Lemma \ref{lemma:cycle} formalizes this observation.

\begin{lemma} \label{lemma:cycle}
    Let $C$ be an arbitrary cycle. Then, $Z_0(C) = 2$.
\end{lemma}

\begin{proof}
    First, we show that a cycle $C$ can be entirely filled by spending two tokens. Consider two arbitrary vertices $u$ and $v$ in $C$, and spend a token on each of them. This results in one or two disjoint unfilled paths, where $u$ and $v$ are adjacent to the endpoints of these paths. By applying Rule 3 and announcing one of the unfilled paths to the oracle, the player can force the unfilled vertex at one end of that path to be filled. Repeating this process iteratively ensures that all vertices in the cycle are eventually filled.

    However, the player cannot fill the entire cycle by spending only one token. If a single vertex is initially filled, it will be adjacent to both endpoints of the remaining unfilled path. Since this unfilled path forms a single component, neither Rule 2 nor Rule 3 can be applied to force additional vertices to be filled. Hence, at least two tokens are necessary.
\end{proof}

Using the same argument, we can derive the following result.

\begin{corollary}
   Let $C$ be a cycle that is an induced subgraph of a cactus graph $G$. In the $q$-analogue of the zero forcing game with $q = 0$, as soon as any two vertices in $C$ are filled (by any rule), the entire cycle $C$ can be filled.
\end{corollary}

Consider a corresponding rooted tree representation of the cactus graph and a subtree of it. Assume that the root of this subtree is a cycle, which we denote by $C$ (we will later consider the case where the root is a single vertex rather than a cycle). To fill the entire cycle $C$, the player must reach a point where at least two vertices of $C$ are filled. 

If any cycle that shares a vertex with $C$ gets filled first, then at least one vertex of $C$ is also filled without spending a token, as they share a vertex. Similarly, if any adjacent vertex of a vertex in $C$ gets filled first, then this vertex in $C$ can be filled using Rule 3. Therefore, once an adjacent vertex of a vertex in $C$ is filled, at least one vertex of $C$ can also be filled without spending a token. The following lemma formalizes this observation.

\begin{lemma}
    Let $G$ be a cactus graph, and let $C$ be an induced subgraph of $G$ that forms a cycle. Suppose there exists a filled vertex $u$ outside of $C$ that is adjacent to a vertex $v \in C$. Then, by using Rule 3 on the unfilled component consisting of $C$ and the filled vertex $u$, the player can fill the vertex $v$.
\end{lemma}

\begin{proof}
    Consider the component that contains $C$. The vertex $u$ cannot have two adjacent vertices in this component because that would imply the existence of two cycles sharing more than one vertex, contradicting the definition of a cactus graph. Therefore, by announcing the component containing $C$ and noting that $u$ has exactly one neighbor in this component, Rule 3 allows $u$ to force $v$ to be filled.
\end{proof}

Therefore, there are three cases to consider for filling the entire cycle $C$, as outlined below.

\begin{itemize}
    \item Two vertices of $C$ are filled by spending two tokens.
    
    \item Either an adjacent vertex of $C$ or a cycle sharing a vertex with $C$ is filled before $C$, causing a vertex of $C$ to be filled without spending a token. Another vertex of $C$ is then filled by spending a token. The adjacent vertex or the cycle can be either a parent or a child of $C$ in the tree.
    
    \item Any of the following situations occur before $C$ is filled:  
    \begin{itemize}
        \item Two adjacent vertices of $C$ are filled.  
        \item Two cycles that share a vertex with $C$ are filled.  
        \item An adjacent vertex to $C$ and a cycle that shares a vertex with $C$ are filled.  
    \end{itemize}  
    In these cases, two vertices of $C$ are filled without spending any tokens. The adjacent vertices or cycles can either be the parent and a child of $C$ or two of its children in the tree.
\end{itemize}

Now consider the case where the root of the subgraph is a single vertex. In this case, we only need to determine whether any of its children or its parent is filled before this vertex or if a token is spent on it. The next lemma formalizes this observation.

\begin{lemma}
    In a cactus graph $G$, consider a vertex $v$ that is not part of any cycle. If an adjacent vertex of $v$, such as $u$, is filled before $v$, then by applying Rule 3 on the filled vertex $u$ and the unfilled component containing $v$, the vertex $v$ can be forced to be filled.
\end{lemma}

Therefore, there are two cases to consider when the root of the subgraph is a vertex $v$, as outlined below.

\begin{itemize}
    \item A token is spent on $v$ to fill it.
    \item A neighbor of $v$ is filled before $v$, forcing $v$ to be filled by Rule 3. This neighbor can be either a parent in the tree or a child.
\end{itemize}

By using the previous explanations and considering all the discussed cases for each vertex in the tree, we can compute the minimum number of tokens required.

Consider a rooted tree representation of the cactus graph. Our goal is to compute the minimum number of tokens needed to fill all vertices in each subtree rooted at a vertex in the tree. This can be achieved by performing the calculation from the leaves up to the root. The calculation for the leaves is straightforward, as we can directly apply the cases discussed.

Now, consider a vertex $w$. We already know the minimum number of tokens required to fill the subtree rooted at any of its children in each of the discussed cases. Using these values, we can compute the minimum number of tokens needed to fill the subtree rooted at $w$ in any scenario.

To determine the desired value, we consider all possible rooted trees of the cactus graph by selecting each vertex of $G$ as the root. By performing the following calculations for all trees and selecting the minimum value among them, we obtain $Z_0(G)$. Below, we outline how to compute the minimum number of tokens needed to fill a rooted tree of a given graph.

In the following, a \textit{free filled vertex} refers to a vertex that is forced to be filled by another vertex, rather than by spending a token on it. Next we present a recurrence relation for computing $Z(G)$, where $G$ is a block graph.\\

For $ i \in \{0,1\}, j \in \{0,1,2\} $, we define:
\begin{multline*}
   Z_0(C, i, j) = \text{the number of tokens needed to fill the subtree rooted at } \\ C \text{ if it gets } i \text{ free filled vertex from the parent and } j \text{ free vertices from the children.} 
\end{multline*}

\allowdisplaybreaks

Let $ C $ be a cycle. We have:
\begin{align*}
 Z_0(C,0,0) &= 2 + \sum_{S \in \text{Children of } C} \min_{k \in \{0,1,2\}} Z_0(S, 1 , k), \\\\
 Z_0(C,0,1) &= 1 + \min_{H \in \text{Children of } C}  \sum_{\substack{S \in \text{Children of } C \\ S \neq H}} \min_{k' \in \{0,1,2\}} Z_0(S, 1 , k') + \min_{k \in \{0,1,2\}} Z_0(H, 0 , k), \\\\
 Z_0(C,0,2) &= \min_{\substack{H , Q \in \text{Children of } C \\ H \neq Q}}\sum_{\substack{S \in \text{Children of } C \\ S \neq H , Q}} \min_{k'' \in \{0,1,2\}} Z_0(S, 1 , k'') + \\
 &\min_{k' \in \{0,1,2\}} Z_0(H, 0 , k') + \min_{k \in \{0,1,2\}} Z_0(Q, 0 , k), \\\\
 Z_0(C,1,0) &= 1 + \sum_{S \in \text{Children of } C} \min_{k \in \{0,1,2\}} Z_0(S, 1 , k), \\\\
 Z_0(C,1,1) &= 1 + \min_{H \in \text{Children of } C} \sum_{\substack{S \in \text{Children of } C \\ S \neq H}} \min_{k' \in \{0,1,2\}} Z_0(S, 1 , k') + \min_{k \in \{0,1,2\}} Z_0(H, 0 , k).
 \end{align*}

If $ C $ is a singular vertex, we have:
\begin{align*}
 Z_0(C,0,0) &= 1 + \sum_{S \in \text{Children of } C} \min_{k \in \{0,1,2\}} Z_0(S, 1 , k), \\\\
 Z_0(C,0,1) &= \min_{H \in \text{Children of } C}  \sum_{\substack{S \in \text{Children of } C \\ S \neq H}} \min_{k' \in \{0,1,2\}} Z_0(S, 1 , k') + \min_{k \in \{0,1,2\}} Z_0(H, 0 , k), \\\\
 Z_0(C,1,0) &= \sum_{S \in \text{Children of } C} \min_{k \in \{0,1,2\}} Z_0(S, 1 , k).
 \end{align*}
 \medskip

By considering each cycle or singular vertex as the root of a tree, computing these values for all such trees, and taking the minimum of $Z_0(C, 0, k)$ for $k \in \{0,1,2\}$ with $C$ as the root of the tree, we obtain the desired value $Z_0(G)$.
In Appendix \ref{apx: code cactus}, we provide the pseudocode for calculating $Z_0$ for cactus graphs. This is an implementation of the approach described above, and it has a time complexity of $O(n^2)$. Also, the Python code implementing this algorithm for cactus graphs can be found in this \href{https://github.com/sanazmgh/Zq-Forcing-Number-Algorithms}{GitHub repository}.

\section{Exact \texorpdfstring{\(Z_q\)}{Zq}-forcing number for some graph classes}\label{sec:exactvaluegraphclasses}

The exact value of $Z_q$ is known for certain graph classes. Instances of it are paths \cite{butler2020properties}, Zigzag graphs and Comb graphs \cite{butler2020properties}, trees when $q=1$ \cite{butler2020properties}, forests of star graphs \cite{blanco2023z_q}, caterpillar cycles when $q \leq 3$ \cite{blanco2023z_q}, complete bipartite and multipartite graphs \cite{fallat2024q}, and some types of Kneser graphs \cite{fallat2024q}. Also, there exist several bounds for $Z_q$, such as an upper bound for trees and upper bounds for caterpillar cycles by  by Blanco et al. \cite{blanco2023z_q}, and lower bounds for strongly regular graphs when $q \leq 1$ by Fallat et al. \cite{fallat2024q}. 
Here we determine the $Z_q$ for several families of graphs, extending known results about trees from \cite{blanco2023z_q} and \cite{butler2020properties}.



\subsection{Generalised star graphs}

A \textit{generalized star graph} (also called \textit{spider graph}) is a graph $G$ that consists of a central vertex $c$ (also known as the \textit{center}) and a set of disjoint paths $P_1, P_2, \ldots, P_k$, called \textit{paths}, each of which is attached to the central vertex at one of its endpoints and they may vary in length.

Let $H$ be a generalized star graph. For $q = 0$, as mentioned in \cite[Theorem 4]{butler2020properties}, $Z_0(H) = 1$. In this case, the optimal strategies for the oracle and the player are quite obvious. The oracle has only one strategy, which is choosing the only component selected by the player. The player spends one token on the center of the star. If there are still unfilled components (which are all in the form of a path), the player selects one and announces it to the oracle. After the oracle selects the chosen path, the player can apply the forcing rule on at least one of the ends of the path. Then the player can fill the rest of the vertices on the path using Rule 2. This strategy is optimal for the player, as filling out a graph requires at least one token.

Otherwise, if $q > 0$, we provide an optimal strategy for both the oracle and the player.

\begin{remark} 
It is optimal for the player to spend the tokens only on the endpoints of the paths connected to the center. Assume in the optimal strategy, the player spends a token on a node that is not an endpoint in $H$. Consider the path that contains this node (if it is the center, we can assume any path). If the player had spent the token at the endpoint of this path, the rest of the nodes in the path, including the center, could be filled using Rule 2. Therefore, it is optimal to spend the tokens only on the endpoints of paths (the leaves).
\end{remark}

We are now ready to provide the optimal strategies:

\begin{itemize}
\item \textbf{Optimal strategy for the oracle:} If two of the components announced by the player share a filled neighbor, the oracle chooses these two components. Otherwise, if there were a component which the center was not its neighbor, the oracle chooses that. Elsewhere, it chooses one randomly.
\item \textbf{Optimal strategy for the player:} If $H$ has $k$ leaves, then the player’s optimal strategy is to spend $k-1$ tokens on the leaves and then fill the rest of the vertices using Rule 2. This is because a path cannot be filled without spending any token on it, unless by using forcing rule on the center and filling the vertices on the path. However, this only happens if the center has only one uncolored neighbor.
\end{itemize}

If $G$ is a forest of generalized star graphs, then the strategies above are very similar to the approach by Blanco et al.\cite{blanco2023z_q}. Note that we must spend at least one token on each of the components since we cannot apply Rule 2 or Rule 3 without having any filled vertex. Afterward, since filling a vertex connected to the center allows filling the rest of the vertices on that path, the approach is the same as in \cite{blanco2023z_q}.

\subsection{Block graphs with cliques of size greater than two} \label{subsec: blockgraphs}

In Section \ref{subsec:polyalgoblckgraphs} we introduced block graphs and provided an algorithm for computing the zero forcing number of such graphs. Here we extend that result and prove that for any arbitrary $q$, $Z_q(G) = Z(G)$ when all blocks in the block graph have size greater than two.

\begin{theorem} \label{block_graphs_gen}
If $G$ is a block graph with cliques of size greater than two, then for any arbitrary $q$, it holds that $Z_q(G) = Z(G)$.
\end{theorem}

\begin{proof}
    We prove this by induction on the number of blocks in the graph. The base case is a graph with one block. Assume that the block has size $\eta > 2$. In this case, since the graph is a complete graph, to apply Rule 2 or Rule 3, we need to spend a token on $\eta-1$ vertices. After filling $\eta-1$ vertices, the remaining unfilled vertex can be filled using Rule 2. Hence, in this case, we have $Z_q(G) = Z(G)$. Now, assume that $G$ is a block graph with $k$ blocks, and assume the induction hypothesis holds: $Z_q(G) = Z(G)$ for any block graph with $k-1$ blocks whose block sizes are greater than two. Consider a block where only one of its vertices is shared with other blocks (such a block must exist). Let $v$ be the shared vertex, and let $U = \{u_1, u_2, \dots, u_{\eta-1} \}$ represent the other vertices in the block. Now, consider an optimal strategy for the player. There are two cases:

\begin{itemize}
    \item \textbf{In the optimal strategy $v$ is filled before $u_1, u_2, \dots, u_{\eta-1}$:} 
    If $v$ is filled before the other vertices in $U$, none of the vertices in $U$ can be filled using Rule 3 or Rule 2 after $v$ gets filled, as their only filled neighbor is $v$, and all vertices in $U$ belong to the same component. Furthermore, none of the vertices in $U$ can be filled using Rule 2 or Rule 3 unless $\eta-2$ vertices in $U$ are filled, because they all remain in the same component. Once all but one vertex in $U$ are filled, the remaining vertex can be filled by Rule 2. Thus, an optimal strategy exists where Rule 3 is not applied to any vertices in $U$. In this optimal strategy, if Rule 3 is used, it must be applied to vertices in $V(G) \setminus U$, which is a graph with $k-1$ blocks. However, by the induction hypothesis, we know there exists an optimal strategy for filling $V(G) \setminus (U \setminus \{v\})$ without using Rule 3. 
    
    Note that we have proved that we need to spend $\eta - 2$ tokens in $U$, and we are considering the case where $v$ gets filled before any vertex in $U$. Therefore, from any optimal strategy where $v$ gets filled before $U$, we can simply obtain an optimal strategy for the induced subgraph on vertices $V \setminus U$ by only considering the steps that are applied to these vertices in the given optimal strategy. This also hols the other way around. For any optimal strategie for the induced subgraph on vertices $V \setminus U$, we can extend this optimal strategie to an optimal stratgy for graph $G$ by spending $\eta - 2$ tokens in vertices $U$ and consider the whole block filled imediatly after $v$ gets filled in the primal optimal strategy. Thus, by induction, in this case there is an optimal strategy that does not require applying Rule 3 to $V(G) \setminus U$ or to the vertices in $U$. Hence, we have proved that $Z_q(G) = Z(G)$ for a block graph $G$ with $k$ blocks.

    \item \textbf{Any other case:} We show that any optimal strategy can be transformed into another optimal strategy in which $v$ gets filled sooner than the vertices in $U$. Note that the vertices in $U$ cannot force any vertex outside $U \cup \{v\}$ to be filled, since they are not neighbors of any such vertex. The only vertex in the considered block that can force a vertex outside $U$ to be filled using Rule 2 or 3 is $v$, which requires $v$ to be filled at first. Therefore, as long as the vertices in $U$ do not force $v$ to be filled, we can assume that all the vertices in $U$ get filled immediately after $v$, i.e., all the steps in the optimal strategy involving vertices in $U$ can be reordered to occur after filling $v$, following the same argument as in the previous case. Thus, we have shown that if $v$ gets filled independently from the vertices in $U$, our argument holds, and we can transform the initially optimal strategy into a strategy that fills $v$ sooner than any vertex in $U$. The only challenge is when $v$ gets forced by a vertex in $U$. For this to happen, all the vertices in $U$ must have already been filled, since $v$ and any unfilled vertex in $U$ are in the same component. Notice that to fill all the vertices in $U$ before filling $v$, there is no way other than spending $\eta - 1$ tokens on the vertices in $U$. Instead, we could simply assume $v$ is filled using a token, and then another arbitrary vertex in $U$ is forced to be filled. This, as discussed before, can be considered as a previous case. 
    
    Therefore, no matter what optimal strategy we are considering, we can always transform it into another optimal strategy that satisfies the conditions of the previous case, ensuring that all the statements in the previous case also hold in other cases.
    
\end{itemize}

Thus, we have proved that there exists an optimal strategy that does not use Rule 3, and therefore $Z_q(G) = Z(G)$.
\end{proof}

Now, we provide the optimal strategies for both the player and the oracle.
\begin{itemize}
    \item \textbf{Oracle's strategy:} The oracle acts to minimize the player's advantage. Specifically, if there is a component containing a vertex that can be filled using Rule 2 (and no other vertex in this component can be filled using Rule 3 after the union with filled vertices), the oracle announces this component. Additionally, if there is a nonempty subset of components such that the union of these components with the filled vertices leaves any filled vertex with either no unfilled neighbors or more than two unfilled neighbors, the oracle announces this set. Otherwise, the oracle announces the component with the smallest size.

    \item \textbf{Player's strategy:} The player can use an optimal strategy that, by Theorem \ref{block_graphs_gen}, does not involve the use of Rule 3. Therefore, the strategy is to identify $Z(G)$ vertices and fill them using Rule 1. This has already been covered in Section \ref{subsec:polyalgoblckgraphs}. The player can simply run the algorithm stated in Theorem \ref{Theorem: Z algorithm block graph} to determine the vertices that need to be filled.
\end{itemize}

\subsection{Generalized windmill graphs}

A \emph{windmill graph}, denoted by $W(\eta, k)$, consist of $\eta$ copies of the complete graph $K_k$, with every node connected to a common node. Two generalization of these graphs, recently introduced by Kooij \cite{kooij2019generalized}, are the so called generalized windmill graphs. For both generalization we replace the central node, connecting all $\eta$ copies of the complete graph $K_k$, by $l$ central nodes. For the first generalization, we assume the $l$ central nodes are all connected, i.e. they form a clique $K_l$. We
call this a \emph{generalized windmill graph of Type I} and denote it by $W'(\eta,k,l)$. Note that $W'(\eta,k,1)=W(\eta,k)$. For the second generalization, we assume the $l$ central nodes have no connections among each other. We will refer to it as a \emph{generalized windmill graph of Type II} and denote it by $W''(\eta,k,l)$. Moreover, for a generalized windmill graph of Type I, $ W'(\eta, k, l) $, and a generalized windmill graph of Type II, $ W''(\eta, k, l) $, we define the set $ C $ as the set of $ l $ central vertices and the set $ Q $ as the set of $ \eta \times k $ vertices in the $ \eta $ complete graphs $ K_k $.


\begin{theorem} \label{Theorem: windmill I}
    Let $ G := W^{\prime}(\eta, k, l) $, $k>1$, be a generalized windmill graph of Type I. Then $ Z_q(G) = Z(G) = \eta \times (k-1) + l$.
\end{theorem}

\begin{proof}
    Consider an optimal strategy for the player. We will show that whenever Rule 3 is used by the player, Rule 2 can be used instead. Thus, we provide a strategy that requires the same number of tokens without using Rule 3. This proves $Z_q(G) = Z(Q)$.

    Suppose that, in an optimal strategy, there exists a step where vertex $ v $ can force vertex $ u $ to be filled by the player using Rule 3. There are three possible cases.
    \begin{itemize}
        \item \textbf{$v \in Q$ and $u \in Q \cup C$:} Let $ Q' $ be the clique that contains $ v $. Note that if $ u \in Q $ is forced by $ v $ using Rule 3, both vertices must be in the same clique (as they would otherwise be disconnected). If $ u \in C $, since all vertices in $ C $ are connected to all vertices of $ G $, the vertices in $ Q' \cup C $ also form a clique. For $ v $ to force $ u $ to be filled, all other vertices in $ Q' \cup C $ (i.e., all vertices except $ u $) must already be filled. This is because if there is an unfilled vertex $ w $ in $ Q' \cup C  \setminus \{u\}$, then by the definition of a generalized windmill graph of Type I, $ w $ would be a neighbor of both $ u $ and $ v $. Thus, regardless of the player’s choice of components or the oracle’s selection, $ w $ would remain in the same component as $ u $, and $ v $ would not be able to force $ u $ to be filled. Therefore, all vertices in $ Q' \cup C \setminus \{u\} $, must already be filled. If this condition holds, $ v $ could fill $ u $ using Rule 2 instead, as all of $v$’s other neighbors are already filled.

        \item \textbf{$v \in C$ and $u \in Q$:} Let $ Q' $ be the clique that contains $u$. For $ v $ to force $ u $ to be filled, all other vertices in $ Q' \cup C $ (i.e., all vertices except $ u $) must already be filled. This is because if there is an unfilled vertex $ w $ in $ Q' \cup C  \setminus \{u\}$, then by the definition of a generalized windmill graph of Type I, $ w $ would be a neighbor of both $ u $ and $ v $. Thus, regardless of the player’s choice of components or the oracle’s selection, $ w $ would remain in the same component as $ u $, and $ v $ would not be able to force $ u $ to be filled. Therefore, all vertices in $ Q' \cup C \setminus \{u\} $, must already be filled. If this condition holds, another vertex in $Q'$ could fill $ u $ using Rule 2 instead, as all of it’s other neighbors are already filled.

        \item \textbf{$v \in C$ and $u \in C$:} For $ v $ to force $ u $ to be filled using Rule 3, all vertices in $ Q \cup C \setminus \{u\} $ must already be filled. This is because if there is an unfilled vertex $ w $ in $ Q \cup C \setminus \{u\} $, by the definition of a generalized windmill graph of Type I, $ w $ would be a neighbor of both $ v $ and $ u $, leading to the same conclusion as in the previous cases.
    \end{itemize}

    We have shown that in any optimal strategy, if a vertex is filled using Rule 3, it could instead be filled using Rule 2. Therefore, we have provided a strategy that uses the same number of tokens (as Rule 1 was not used in any step of the proof) and thus constitutes an optimal strategy without relying on Rule 3. This proves that $Z_q(G) = Z(Q)$.
    \\

    Now we need to prove that $ Z(G) = \eta \times (k-1) + l $. By following the steps used to prove $ Z_q(G) = Z(G) $, this equation is straightforward to observe. Note that, as established above, we cannot force a vertex $ v $ to be filled by applying Rule 2 if any of the following cases occur:

    \begin{itemize}
        \item $ v \in Q $ and another vertex in the same clique as $ v $ is unfilled.
        \item $ v \in Q $ and a vertex in $ C $ is unfilled.
        \item $ v \in C $ and a vertex in $ C $ is unfilled or at least one vertex in each clique is unfilled.
    \end{itemize}

    Notice that to force a vertex in $ C $, all the other vertices in $C$ and all the vertices in at least one clique must already be filled. Similarly, to force a vertex in $ Q $, all vertices in $ C $ and all other vertices in the clique containing $ v $ must be filled. Therefore, the optimal strategy is to spend a token on all vertices except for one arbitrary vertex in each clique, requiring exactly $ \eta \times (k-1) + l $ tokens. This proves that $ Z_q(G) = Z(G) = \eta \times (k-1) + l $.
\end{proof}

\begin{theorem} \label{Theorem: windmill edge case}
    Let $ G := W^{\prime}(\eta, k, l) $, $k = 1$, be a generalized windmill graph of Type I. Then

    $$Z_q(G) = 
    \begin{cases}
        l & \text{for } \eta = 1, \\
        l + \eta - 2 & \text{for } q > 0 \text{ and } \eta > 1, \\
        l & \text{for } q = 0.
    \end{cases}$$
\end{theorem}

\begin{proof}
We provide proofs for the three cases separately.

\begin{itemize}
    \item \textbf{$\eta = 1$, $q > 0$:}  
    In this case, $W'(1,1,l) = K_{l+1}$. To fill the entire clique, the player must spend tokens on all vertices except one; the last unfilled vertex can then be forced by Rule~2, using any filled neighbor (since all other vertices are already filled).  
    To see why fewer tokens do not suffice, suppose two vertices $u$ and $v$ remain unfilled. Neither can be forced by Rule~2 or Rule~3, since each is adjacent to every other vertex, including each other. Thus, for a clique of size $l+1$, exactly $l$ tokens are required, regardless of the value of $q$.

    \item \textbf{$\eta > 1$, and $q > 1$:}  
    We first show that the whole graph cannot be filled using only $l + \eta - 3$ tokens, and then present a strategy using $l + \eta - 2$ tokens.

    Assume there exists a strategy using at most $l + \eta - 3$ tokens. By the pigeonhole principle, either two vertices in $C$ receive no token, or two vertices in $Q$ receive no token. There are two cases:

    \begin{itemize}
        \item If two vertices $u,v \in C$ are unfilled, then neither can be forced by Rule 2 or Rule 3, as they are adjacent to each other and to every other vertex. \\ 

        \item If two vertices $u,v \in Q$ are unfilled, then since $k=1$, they must each be forced by a vertex in $C$. But if one vertex of $C$ is also unfilled, then all unfilled vertices lie in the same component, making Rule 3 unusable. Meanwhile, no vertex in $C$ can force $u$ or $v$ using Rule 2, since each such vertex has at least two unfilled neighbors in $Q$. If instead all unfilled vertices lie in $Q$, then because $q>1$, the oracle selects multiple components, ensuring that each vertex in $C$ has at least two unfilled neighbors in $Q$. Again, Neither Rule 2 nor Rule 3 can be applied.  
    \end{itemize}
    In both cases we reach a contradiction. Hence, at least $l+\eta-2$ tokens are required.\\\\
    Now we show the sufficiency of spending $l+\eta-2$ tokens. First, the player fills all vertices except one in $C$ and one in $Q$, which requires $l+\eta-2$ tokens. Since $\eta > 1$, at least one vertex in $Q$ is filled; let us call it $v$. The unfilled vertex in $C$ is the unique unfilled neighbor of $v$, so $v$ forces it by Rule 2 to be filled. Finally, with all vertices of $C$ filled, the remaining unfilled vertex in $Q$ can also be forced to be filled using Rule 2. Thus the graph is completely filled using $l + \eta - 2$ tokens, proving that $Z_q(G) = l + \eta - 2$.

    \item \textbf{$\eta > 1$, $q = 0$:}  
    We again show that fewer than $l$ tokens are insufficient, and then present a strategy with $l$ tokens.  

    Suppose the player uses only $l-1$ tokens. Then at least one vertex $v \in C$ is unfilled. Consider how $v$ might be forced. There are two cases:

    \begin{itemize}
        
        \item If a vertex $u \in C$ were to force $v$, there would exist another unfilled vertex (since $l-1 < l+\eta-1$). This additional unfilled vertex would also be adjacent to both $u$ and $v$, preventing $u$ from forcing $v$ to be filled using Rule 2 or Rule 3.\\  
    
        \item If a vertex $u \in Q$ were to force $v$, then, since at most $l-1$ tokens were used, at least two vertices in $C$ would remain unfilled when $u$ attempts to force $v$. As these vertices are both adjacent to $u$ and to each other, $u$ cannot apply Rule 2 or Rule 3 to force $v$ to be filled.
    \end{itemize}
    In either case, $v$ cannot be forced, so $l-1$ tokens do not suffice.\\\\   
    For sufficiency, the player spends $l$ tokens to fill all vertices of $C$. Since $q=0$, the oracle has no choice in Rule 3. At each step, the player selects an unfilled vertex of $Q$ (note that $Q$ is independent since $k=1$). The oracle must accept this choice, and the vertex is forced by Rule 3 using a vertex in $C$. Repeating this process fills all of $Q$. Thus, the graph is completely filled using $l$ tokens, so $Z_q(G) = l$ in this case.
\end{itemize}

Therefore, in all three cases we obtain the equality stated in the theorem.
\end{proof}


\begin{theorem}
    Let $G := W''(\eta,k,l)$ be a generalized windmill graph of Type II and $k, \eta, l>1$. Then
    $$Z_q(G) = 
    \begin{cases}
        \min \{\eta \times (k-1) + l , \eta \times k\} & \text{if } q = 0, \\
        \eta \times (k-1) + l & \text{for } q \geq 1.
    \end{cases}$$
\end{theorem}

\begin{proof}
    Notice that if vertex $v$ can force a vertex $u$ to be filled using Rule 3 or Rule 2, either both of them should be in a same clique or one of them is $Q$ and the other is $C$. In other cases since $v$ and $u$ can not be neighbors, $v$ can not force $u$ to be filled.
    \\
    Now, we want to consider all possible cases of placement $v$ and $u$ and show if $q \neq 0$, $Z_q(G) = Z(G) = \eta \times (k-1) + l$, otherwise $Z_q(G) = \min \{\eta \times (k-1) + l , \eta \times k\}$. \\

    Suppose that, in an optimal strategy, there exists a step where vertex $ v $ can force vertex $ u $ to be filled by the player using Rule 3.

    \begin{itemize}
        \item \textbf{$v$ and $u$ in a same clique:} Let $Q'$ be the clique which $u$ and $v$ are in. For $v$ to be able to force $u$ to be filled by Rule 3, all the vertices in $Q' \setminus \{u\} \cup C$ must be already filled. This is because if there is an unfilled vertex $ w $ in $ Q' \cup C  \setminus \{u\}$, then by the definition of a generalized windmill graph of Type II, $ w $ would be a neighbor of both $ u $ and $ v $. Thus, regardless of the player’s choice of components or the oracle’s selection, $ w $ would remain in the same component as $ u $, and $ v $ would not be able to force $ u $ to be filled. Therefore, all vertices in $ Q' \cup C \setminus \{u\} $, must already be filled. If this condition holds, $ v $ could fill $ u $ using Rule 2 instead, as all of $u$’s neighbors are already filled.

        \item \textbf{$v \in C$ and $u \in Q$:} 
        Let $ Q' $ be the clique that contains $ u $. For $ v $ to be able to force $ u $ to be filled using Rule 3, all vertices in $ Q' \setminus \{u\} $ must already be filled. This is because if there exists an unfilled vertex $ w $ in $ Q' \setminus \{u\} $, then by the definition of a generalized windmill graph of Type II, $ w $ would be a neighbor of both $ u $ and $ v $. Thus, regardless of the player’s choice of components or the oracle’s selection, $ w $ would remain in the same component as $ u $, preventing $ v $ from forcing $ u $ to be filled. Notice that if there exists a vertex $ v^{\prime} \in C $ that is not filled, then there is only one component of unfilled vertices. This is because all unfilled vertices in $ Q $ (including $ u $) are neighbors of $ v^{\prime} $, and all unfilled vertices in $ C $ are neighbors of $ u $. Therefore, if such a vertex $ v^{\prime} $ exists, all unfilled vertices belong to the same component. In this scenario, the graph on which the player can apply Rule 3 is equivalent to the original graph, meaning there is no distinction between Rule 3 and Rule 2. Now, let us assume there are no unfilled vertices in $ C $. Since we have already established that all vertices in $ Q' \setminus \{u\} $ must be filled, instead of filling vertex $ u $ using Rule 3, $ u $ can be forced to be filled by another vertex in $ Q' $ using Rule 2. Thus, if Rule 3 can be applied to fill vertex $ u $, it can be similarly done using Rule 2.

        \item \textbf{$v \in Q$ and $u \in C$:} 
        Assume there is an unfilled vertex $ w \in Q $. By the definition of a generalized windmill graph of Type II, $ w $ is a neighbor of all unfilled vertices in $ C $ (including $ u $), and $ u $ is a neighbor of all unfilled vertices in $ Q $. In this scenario, the graph on which the player can apply Rule 3 is equivalent to the original graph, meaning there is no distinction between Rule 3 and Rule 2. Now, let us assume all the unfilled vertices are in $ C $. For $ q \geq 1 $, the oracle can always choose any two of the components given by the player (note that each of these components consists of a single vertex in $ C $). Since all vertices in $ Q $ are connected to both of these chosen components, Rule 3 cannot be applied to any vertex. Therefore, we have shown that for $q \geq 1$, if Rule 3 is applied at any step, it can be replaced with Rule 2. However, for $ q = 0 $, the player can select one vertex (component) from $ C $ at each step, and the oracle would have no choice but to announce this given component. In this case, the player can apply Rule 3 to force the chosen vertex to be filled by any vertex in $ Q $.
    \end{itemize}

    We have shown that for $q \geq 1$, any optimal strategy for the player can be transformed into a strategy that uses only Rule 2 while spending the same number of tokens (since Rule 1 was not considered in any of the mentioned scenarios). Therefore, the new version of the strategy remains optimal, and we have shown that $Z_1(G) = Z(G)$. Now we prove that $Z(G) = \eta \times (k-1) + l - 1$.

    We have shown that when applying Rule 2 to any vertex in $Q$, all the vertices in its clique must already be filled. Therefore, we must spend $k-1$ tokens for each clique to color all its vertices except one, which concludes that we must at least spend $\eta(k-1)$ tokens in $Q$. Additionally, to force a vertex in $C$ to be filled, we have shown that all the other vertices in $C$ must already be filled, which concludes that we must spend at least $l-1$ tokens in $C$. 
    
    Now, we can assume that all the vertices in $C$ except one, and all the vertices in each clique except one, are filled. It can be shown that in this state, none of the unfilled vertices can be forced to be filled. If we spend one more token to fill an unfilled vertex in $Q$, all the vertices in one clique, such as $Q'$, get filled. Now, any vertex in $Q'$ can force the unfilled vertex in $C$ to be filled, and after that, all the unfilled vertices in each clique can be forced to be filled by a vertex in its clique. Therefore, we have proved that $Z(G) = \eta \times (k-1) + l$.
    
    For the $q = 0$ case, due to the previous arguments, it is sufficient to consider strategies where Rule 3 is only used to force vertices in $C$ when all the vertices in $Q$ are already filled. We have already shown that to force a vertex in $Q$, all the other vertices in its clique must already be filled. Therefore, we must spend at least $k-1$ tokens in each clique, which sums up to $\eta \times (k-1)$.
    
    After filling all the vertices except one in each clique, assume an arbitrary unfilled vertex $v \in Q$. (Note that in order to apply Rule 3, we need to reach the state where all of $Q$ is filled, and we are only considering strategies that use Rule 2 to fill the vertices in $Q$.) To fill this vertex, we have two possibilities: either we force it to be filled using a vertex in its clique or we fill it using a vertex in $C$. 
    
    If we force it by a vertex in its clique, as we have shown before, we need the whole set $C$ to be filled. In this case, we do not need to apply Rule 3 anymore, and we end up using $Z(G) = \eta \times (k-1) + l$ tokens.
    
    If we force it by a vertex in $C$, we need $Q \setminus \{v\}$ to be filled. Therefore, we need $k$ more tokens: $k-1$ tokens to fill the cliques, and one token to fill one vertex in $C$ to use it to force $v$ to be filled. In this case, we can simply assume we have filled all the vertices in $Q$ by spending tokens on them, which is $\eta \times k$ tokens.

    After filling all the vertices in $Q$, the player can use Rule 3 and choose one vertex at a time as its chosen component. The oracle would announce the same vertex, which can be filled using any vertex in $Q$, and this would use $\eta \times k$ tokens. 

    Therefore, we have proved the desired result.
\end{proof}

\begin{remark} \label{obs: edge case windmill}
    Let $G := W''(\eta, k, l)$ be a generalized windmill graph of type II. If $\eta = 1$, then $W''(1, k, l) = W'(l, 1, k)$, whose exact value of $Z_q$ is given in Theorem \ref{Theorem: windmill edge case}. If $k = 1$, then $W''(\eta, 1, l) = K_{\eta, l}$, and the exact value of $Z_q$ for complete bipartite graphs is provided in \cite[Section~4.2]{fallat2024q}. Finally, if $l = 1$, then $W''(\eta, k, 1) = W'(\eta, k, 1)$, which is covered in \ref{Theorem: windmill I}.
\end{remark}

\subsection*{Acknowledgements}
Aida Abiad is supported by NWO (Dutch Research Council) through the grants \linebreak VI.Vidi.213.085 and OCENW.KLEIN.475. The authors thank the referees, whose feedback helped improving the paper.



\newpage
\section{Appendix}

\subsection{A faster algorithm for computing \texorpdfstring{\(Z_q\)}{Zq} for block graphs} \label{apx: faster zq block graphs}

In Theorem~\ref{Theorem: Z algorithm block graph}, the algorithm proceeds as follows: in each round, it identifies a block (clique) with only one shared vertex (using a naive approach, this step has a time complexity of $O(n^2)$), and then the algorithm performs the specified operations on the vertices of the found block. The block is then removed, marking the end of the round, and the process repeats until the graph becomes empty. Since the number of rounds is approximately $O(n)$, the total processing time using the naive approach is $O(n^3)$. Below, we propose an improved approach that reduces the overall time complexity from $O(n^3)$ to $O(n+m)$.

To compute $Z_q$ for block graphs, we first identify the blocks of the graph and then provide an ordering of these blocks, so that we avoid repeatedly searching for an appropriate block during the computation. The following algorithm can be used to identify the blocks efficiently and provide an ordering of blocks.

For a vertex $v$, we use $\mathrm{disc}[v]$ to denote its discovery time in a DFS traversal, i.e., the timestamp at which $v$ is first visited. We let $\mathrm{low}[v]$ denote the smallest discovery time reachable from $v$, considering $v$ itself, its descendants in the DFS tree, its parent, and possibly its ancestors via a back edge.

\begin{algorithm}[H]
\caption{Find-Blocks($G$)}
\label{alg:find-cliques}
\begin{algorithmic}[1]
\Require Connected block graph $G = (V,E)$
\Ensure Set of maximal cliques $\mathcal{C}$ (the blocks)
\State $S \gets \varnothing $ \Comment{Stack of edges}
\State $\mathcal{C} \gets \varnothing$ \Comment{Queue of blocks}
\State Initialize $\text{disc}[v] \gets -1$, $\text{low}[v] \gets -1$ for all $v \in V$
\State $\text{time} \gets 0$
\State \Call{DFS}{$v$, \text{parent} = NIL} \Comment{for an arbitrary vertex $v \in V$)}

\State \Return $\mathcal{C}$
\end{algorithmic}
\end{algorithm}

\begin{algorithm}[H]
\caption{DFS($v$, parent)}
\label{alg:dfs}
\begin{algorithmic}[1]
\State $\text{time} \gets \text{time} + 1$
\State $\text{disc}[v] \gets \text{time}$, $\text{low}[v] \gets \text{time}$
\State $S.\text{push} ((v,u))$
\For{each neighbor $u$ of $v$}
    \If{$\text{disc}[u] = -1$} \Comment{$u$ not visited}
        \State \Call{DFS}{$u$, $v$}
        \State $\text{low}[v] \gets \min(\text{low}[v], \text{disc}[u])$
        
        \If{$\mathrm{low}[u] \geq \mathrm{disc}[v]$}
            \State Pop edges from $S$ until $(v,u)$ is removed.
            \State $B \gets \{ x \in V : x \text{ incident to popped edges}\}$
            \State $\mathcal{C}.\text{push}(B)$
    \EndIf
    \ElsIf{$u \neq \text{parent}$} \Comment{Back edge}
        \State $\text{low}[v] \gets \min(\text{low}[v], \text{disc}[u])$
    \EndIf
\EndFor
\end{algorithmic}
\end{algorithm}

In a DFS traversal of a block graph, line~8 of Algorithm~\ref{alg:dfs} occurs when $v$ is the first vertex of its clique to be discovered; the traversal then visits all vertices of this clique before returning to $v$. Algorithm~\ref{alg:find-cliques} records the blocks in exactly the order in which the cliques are identified, the discovery proceeds in a DFS-like finish order: a leaf clique along with its siblings is explored first, followed by their parent clique, then the leaves of the parent’s sibling cliques, then those siblings, and so on. Hence, the steps of Theorem~\ref{Theorem: Z algorithm block graph} can be applied directly to the blocks in the order given by $\mathcal{C}$. At each step, the block at the top of $\mathcal{C}$ has exactly one shared vertex, so it can be processed and then removed. By construction, the new top block also has a single shared vertex, ensuring that no additional search for a suitable clique is required.

The time complexity of Algorithm \ref{alg:find-cliques}, which calls Algorithm \ref{alg:dfs}, is $O(n+m)$. The DFS itself runs in $O(n+m)$ time, and each edge is added to $S$ once and popped once, while each vertex is added to $\mathcal{C}$ as a member of some set $B$ exactly once. Hence, the overall running time is $O(n+m + n + m) = O(n+m)$. Moreover, as stated in Theorem \ref{Theorem: Z algorithm block graph}, the total number of rounds of the operations is $O(n)$. After this preprocessing, which provides an ordering of the blocks, finding a block with only one shared vertex takes $O(1)$ time. Therefore, the total time complexity for computing $Z$ on a block graph $G$ is $O(n+m) + O(n) = O(n+m)$. The implementation of this pseudocode can be found in this \href{https://github.com/sanazmgh/Zq-Forcing-Number-Algorithms}{GitHub repository}.
We ran the algorithm on several block graphs, and the empirical results are presented in the following table.
\begin{table}[H]
\centering
\begin{tabular}{|c c c c c |} 
 \hline
$n$ & $m$ & No. blocks & running time (s) & $Z_0(G)$ \\
\hline
    10   & 30   & 3    & 0.000039   & 7 \\
    50   & 835   & 7   & 0.000287 & 43\\
    100   & 2912   & 10  & 0.001084 & 90 \\
    500   & 41277   & 15   & 0.011006   & 485 \\
    1000  & 405099   & 6   & 0.127769 & 994 \\
\hline
\end{tabular}
\end{table}

\newpage

\subsection{Pseudocode of the algorithm for computing \texorpdfstring{\(Z_0\)}{Z0} for cactus graphs}\label{apx: code cactus}

In Subsection~\ref{sec:polyalgocactusgraphs}, we introduced a recursive procedure for computing $Z_0$ for a given cactus graph $G$. This procedure considers the tree representation of the cactus graph and evaluates all possible rootings of this tree. For each rooting, the calculation proceeds in the order of the DFS finishing times (i.e., the time when the algorithm finishes processing a vertex), computing the corresponding temporary value of $Z_0$. The minimum over all temporary values of $Z_0$ is then returned as the final result.

Below we provide the pseudocode for the procedure. Here, $\text{val}[v][i][j]$ corresponds to the quantity $Z_0(C,i,j)$ introduced in Subsection~\ref{sec:polyalgocactusgraphs}. For clarity in the pseudocode, we also use $\text{dp}[v][0]$ and $\text{dp}[v][1]$, denoting respectively the number of tokens required when vertex $v$ is not filled by its parent and when it is filled by its parent.


\begin{algorithm}[H]
\caption{Compute $Z_0(G)$}
\label{alg:find-z0}
\begin{algorithmic}[1]
\Require Connected cactus graph $G = (V,E)$
\Ensure The value of $Z_0(G)$
\State $\text{ans} \gets \text{INF}$
\For{each $v\in V$}
    \State Initialize $S \gets \varnothing $ \Comment{Stack of edges}
    \State Initialize $\text{disc}[v] \gets -1$, $\text{low}[v] \gets -1$ for all $v \in V$
    \State $\text{time} \gets 0$
    \State Initialize $\text{val}[u][i][j] \gets -1 $ for all $u \in V$, $i \in \{0,1,2\}$, and $j\in \{0,1\}$
    \State Initialize $\text{dp}[u][i] \gets 0$ for all $u \in V$, and $i \in \{0,1\}$
    \State \Call{DFS}{$v$, \text{parent} = NIL}
    \State $\text{ans} \gets \text{min}(\text{ans}, \text{dp}[v][0])$
\EndFor
\State \Return ans
\end{algorithmic}
\end{algorithm}

\begin{algorithm}[H]
\caption{DFS($v$, parent)}
\label{alg:cycle-tree}
\begin{algorithmic}[1]
\State $\text{time} \gets \text{time} + 1$
\State $\text{disc}[v] \gets \text{time}$, $\text{low}[v] \gets \text{time}$
\State $S.\text{push} ((v,u))$
\State Initialize $\text{min-dp} \gets \text{INF}$
\For{each neighbor $u$ of $v$}
    \If{$\text{disc}[u] = -1$} \Comment{$u$ not visited}
        \State \Call{DFS}{$u$, $v$}
        \State $\text{low}[v] \gets \min(\text{low}
        [v], \text{disc}[u])$
        \If{$\mathrm{low}[u] \geq \mathrm{disc}[v]$}
            \State Initialize $\text{sum} \gets 0 \text{, min1} \gets INf\text{, min2} \gets INF$
            \State Pop edges from $S$ until $(v,u)$ is removed.
            \State $C \gets \{ x \in V : x \text{ incident to popped edges\}}$
            \State $\text{sum} \gets (\sum \text{dp}[x][1] : x \in C \text{, dp}[x][1] \neq -1)$
            \State $\text{min1} \gets \text{minimum of } (\text{dp}[x][0] - \text{dp}[x][1]) \text{, where } x \in C \text{, dp}[x][1] \neq -1)$
            \State $\text{min2} \gets \text{second minimum of } (\text{dp}[x][0] - \text{dp}[x][1]) \text{, where } x \in C \text{, dp}[x][1] \neq -1)$
            \If{$|C|=2$} $C$ \Comment{$C$ is a single-edge block}
                \State $\text{val}[v][0][0]= 1 + \text{sum}$
                \State $\text{val}[v][0][1]= \text{sum} + \text{min1}$
                \State $\text{val}[v][1][0]= \text{sum}$
            \Else \Comment{$C$ is a cycle block} 
                \State $\text{val}[v][0][0]=2+ \text{sum}$
                \State $\text{val}[v][0][1]=1+ \text{sum} + \text{min1}$
                \State $\text{val}[v][0][2]= \text{sum} + \text{min1} + \text{min2}$
                \State $\text{val}[v][1][0]=1+ \text{sum} $
                \State $\text{val}[v][1][1]=\text{sum}+\text{min1} $
    \EndIf
    \State $\text{dp}[v][0] = \text{dp}[v][0] + \min(\text{val}[v][1][0], \text{val}[v][1][1])$   \State $\text{dp}[v][1] = \text{dp}[v][1] + \min(\text{val}[v][1][0], \text{val}[v][1][1])$    \State $\text{min-dp} \gets \min(\text{min-dp}, \min(\text{val}[v][0][i]-\text{val}[v][1][j] : \text{for } i\in \{0,1,2\}, j \in \{0,1\} \text{ where val}[v][1][j] \leq \text{val}[v][1][(j+1)\mod{2}]))$
        \EndIf
    \ElsIf{$u \neq \text{parent}$} \Comment{Back edge}
        \State $\text{low}[v] \gets \min(\text{low}[v], \text{disc}[u])$
    \EndIf
\EndFor

\State $\text{dp}[v][0] \gets \text{dp}[v][0]+ \text{min-dp}$

\end{algorithmic}
\end{algorithm}

In Algorithm~\ref{alg:find-z0}, all variables are first initialized, and then DFS is called from every vertex. In this way, we consider all possible rootings of the graph $G$, compute the temporary value of $Z_0$ for each rooted graph, and finally return the minimum of these values.

In Algorithm~\ref{alg:cycle-tree}, cycles are detected during the DFS whenever a back edge is encountered (line~9). Once a cycle is identified, we compute $\text{val}[v][i][j]$ for its vertices.Since a vertex $v$may be shared by several cycles attached to v in the rooted tree, $\text{val}[v][i][j]$ is computed for every cycle containing $v$, and the results are then aggregated into the variables $\text{dp}[v][0]$ and $\text{dp}[v][1]$.\\
Algorithm~\ref{alg:find-z0} calls Algorithm~\ref{alg:cycle-tree} $O(n)$ times, while all other operations do not affect the overall complexity. Thus, there are $O(n)$ invocations of Algorithm~\ref{alg:cycle-tree}. Each call to Algorithm~\ref{alg:cycle-tree} runs in $O(n+m)$ time, since it performs a DFS with additional $O(1)$-time operations. For cycle detection, observe that the vertices of each cycle are processed once during the DFS, and each edge is pushed onto and popped from the stack $S$ exactly once. Therefore, the running time is $O(n+m)$. In cactus graphs, we have $O(n) = O(m)$, so $O(n+m) = O(n)$. Consequently, the overall time complexity of computing $Z_0$ for a cactus graph is $O(n) \cdot O(n) = O(n^2)$.

The implementation of this pseudocode can be found in this \href{https://github.com/sanazmgh/Zq-Forcing-Number-Algorithms}{GitHub repository}. We ran the algorithm on several cactus graphs, and the empirical results are presented in the following table.

\begin{table}[H]
\centering
\begin{tabular}{|c c c c c |} 
 \hline
$n$ & $m$ & No. cycles & running time (s) & $Z_0(G)$ \\
\hline
    10   & 12   & 3    & 0.000651   & 4 \\
    50   & 65   & 16   & 0.015680 & 17\\
    100   & 132   & 33  & 0.004 & 34 \\
    500   & 20   & 25   & 0.056641   & 26 \\
    1000  & 1295   & 296   & 13.822404 & 297 \\
    2000 & 2530  & 531  & 83.075469 & 532 \\
\hline
\end{tabular}
\end{table}

\end{document}